\documentclass[11pt]{amsart}
\UseRawInputEncoding
\usepackage{amscd}
\usepackage{amsmath, amssymb}
\usepackage{amsfonts}
\usepackage{amsmath}
\numberwithin{equation}{section}

\usepackage{hyperref}

\allowdisplaybreaks

\newtheorem{theorem}{Theorem}[section]
\newtheorem{lemma}{Lemma}[section]

\newtheorem{proposition}[theorem]{Proposition}
\newtheorem{definition}[theorem]{Definition}

\newtheorem{remark}[theorem]{Remark}

\newcommand{\bbm}{\begin{bmatrix}}
\newcommand{\ebm}{\end{bmatrix}}

\DeclareMathOperator{\tr}{tr}

\begin{document}

\title{The Wu-Yau theorem on Sasakian manifolds }


\author{Yong Chen}
\date{}
\address{School of Mathematical Sciences, University of Science and Technology of China, Hefei, 230026, P.R. China}
\email{cybwv988@163.com}

\maketitle

\begin{abstract}
In this note, We proved that a compact Sasakian manifold $(M, \xi, \eta, \Phi, g)$ with negative transverse holomorphic sectional curvature must have has a Sasakian structure $(\xi, \eta^{'}, \Phi^{'}, g^{'})$ with negative transverse Ricci curvature. Similarly, a compact Sasakian manifold with nonpositive transverse holomorphic sectional curvature, then the negative first basic Chern class $-c^{B}_{1}(M, \mathcal{F_\xi})$ is transverse nef and we have the Miyaoka-Yau type inequality. When transverse holomorphic sectional curvature is quasi-negative, we obtain a Chern number inequality.
\end{abstract}

\section{Introduction}
         The study of holomorphic sectional curvature in K\"ahler geometry has been a classical topic, and it attracted many attention in recent years. The recent breakthrough of Wu and Yau \cite{WY1} in which they proved that any projective manifolds with negative holomorphic sectional curvature must have an ample canonical line bundle. Tosatti and Yang \cite{TY} proved that any compact K\"ahler manifold with nonpositive holomorphic sectional curvature must have a nef canonical line bundle and then extended Wu-Yau's work to the K\"ahler case. More recently, Diverio and Trapani \cite{DT} further generalized the result by assuming that the holomorphic sectional curvature is only quasi-negative. In \cite{WY2}, Wu and Yau give a direct proof of the statement that any compact K\"ahler manifold with quasi-negative holomorphic sectional curvature must have an ample canonical line bundle.
         
         Inspired by the Kobayashi conjecture and the above results, it was conjectured that any Hermitian manifolds with quasi-negative(nonpositive) holomorphic sectional curvature must have an ample(nef) canonical line bundle. To the best of author's knowledge, the Wu-Yau's theorem in Hermitian case is still widely open. There are some related works about this problem, readers can refer to \cite{Lee, YZ}.
         
         An odd dimensional Riemannian manifold $(M, g)$ is said to be a Sasakian manifold if the cone manifold $(C(M), \bar{g}) = (M\times \mathbb{R}^{+}, dr^{2} + r^{2}g)$ is K\"ahler. Sasakian geometry was introduced by Sasaki \cite{Sa} in 1960s and is often described as an odd dimensional counterparts of K\"ahler geometry. Sasakian manifold has received a lot of attention because it is the natural intersection of CR, contact and Riemannian geometry, and plays a very important role in Riemannian, algebraic geometry and in physics. Sasakian manifolds first appeared in String theory in \cite{Ma}. Sasaki-Einstein metric is useful in Ads/CFT correspondence. In this paper, we would consider the theorem of Wu-Yau type on Sasakian manifolds. Sasakian manifolds can be studies from may view points as they have many structures. They have a natural foliation structure $\mathcal{F_{\xi}}$, called Reeb foliation, which has a transverse K\"ahler structure; they also has a contact structure. Sasakian geometry is a special kind of contact metric geometry. There exists a unique transverse connection $\nabla^{T}$ corresponding to the Sasakian structure. A good reference on Sasakian geometry can be found in the monograph \cite{BG3} by Boyer and Galicki. We can also define the transverse holomorphic sectional curvature (see section 2). In K\"ahler geometry, The Wu-Yau theorem tell us the negativity of  holomorphic sectional curvature would  lead to the positivity of canonical bundle $K_{M}$, which is equivalent to the positivity of the negative first Chern class $c_{1}(M).$ Similarly, We have the theorems on Sasakian manifold.

%
  
\begin{theorem}\label{Th1}
Let $(M, \xi, \eta, \Phi, g)$ be a 2n + 1 dimensional compact Sasakian manifold with negative transverse holomorphic sectional curvature, then the first basic Chern class $c^{B}_{1}(M, \mathcal{F_\xi})$ is negative. By the transverse Calabi-Yau theorem, $M$ has a Sasakian structure $(\xi, \eta^{'}, \Phi^{'}, g^{'})$ with negative transverse Ricci curvature, which is compatible with $(\xi, \eta, \Phi, g)$.
\end{theorem}

In complex geometry, a (1, 1) class $\alpha \in H^{(1, 1)}(M, \mathbb{R})$ on a K\"ahler manifold $(M, \omega)$ is called nef if for every $\epsilon > 0$ there exists a smooth (1, 1)-form $\theta_{\epsilon} \in \alpha$ on $M$ such that  satisfies
\[ \theta_{\epsilon} \ge -\epsilon\omega. \ \]

We will give the definition of nefness in Sasakian geometry in Section 2 following the definition of nefness in complex geometry.

\begin{theorem}\label{Th2}
Let $(M, \xi, \eta, \Phi, g)$ be a 2n + 1 dimensional compact  Sasakian manifold with nonpositive transverse holomorphic sectional curvature, then the negative first basic Chern class $-c^{B}_{1}(M, \mathcal{F_\xi})$ is transverse nef.
\end{theorem}

\begin{theorem}\label{Th3}
Let $(M, \xi, \eta, \Phi, g)$ be a 2n + 1 dimensional compact  Sasakian manifold with quasi-positive transverse holomorphic sectional curvature, then we have the basic Chern number inequality
\begin{equation}
\int_{M} (- c_{1}^{B}(M, \mathcal{F_{\xi}}))^{n}\wedge \eta > 0.
\end{equation}  
\end{theorem}

There are many results about Miyaoka-Yau inequality for Chern numbers inequality in K\"ahler case; an incomplete list: $K_{X}$ is ample \cite{Yau2}, minimal manifolds of general type \cite{ZYG}, minimal projective varieties \cite{GB}, and compact K\"ahler manifolds whose $c_{1}(K_{X})$ admits a smooth semipositive representative \cite{NR}, compact K\"ahler manifolds with almost nonpositive holomorphic sectional curvature \cite{ZYS}. In Sasakian case, Zhang, Xi proved  the Miyaoka-Yau inequality for basic Chern numbers holds on compact Sasakian-Einstein manifolds \cite{Z2}. We proved the Miyaoka-Yau inequality for basic Chern numbers holds on compact Sasakian manifolds with nonpositive transverse holomorphic sectional curvature. 
\begin{theorem}\label{TH4}
Let $(M, \xi, \eta, \Phi, g)$ be a 2n + 1 dimensional compact Sasakian manifold with nonpositive transverse holomorphic sectional curvature, then we have the following Miyaoka-Yau type inequlity:
\begin{equation}
\int_{M}(2c_{2}^{B}(M, \mathcal{F_{\xi}}) - \frac{n}{n + 1}c_{1}^{B}(M, \mathcal{F_{\xi}})^{2})\wedge(- c_{1}^{B}(M, \mathcal{F_{\xi}}))^{n - 2}\wedge\eta \ge 0
\end{equation}
\end{theorem}

\section{Preliminaries}
\subsection{Preliminary results in Sasakian Geometry}
Sasakian manifolds has many equivalent descriptions. It can be  defined in terms of metric contact geometry or transverse K\"ahler geometry. Boyer and his collaborators \cite{BG1,BG2,BGK,BGM} published a series of papers investigating various differential geometric aspects of Sasakian manifolds. We can find transverse counterparts on Sasakian manifolds of the famous results in K\"ahler manifolds, such as the transverse Calabi-Yau theorem \cite{KA} (see also \cite{BG1,SWZ}), the existence of canonical metrics on Sasakian manifolds \cite{FOW}.  A good reference on Sasakian geometry can be found in the monograph \cite{BG3} by Boyer and Galicki. 

A Sasakian manifold $(M, g)$ has a contact structure $(\xi, \eta, \Phi)$. $\eta\wedge(d\eta)^{n}$ defines a volume element on $M.$ There is a canonical vector field $\xi$ defined by
 \[\eta(\xi) = 1, d\eta(\xi, *) = 0 \ \]
$\xi$ is called the Reeb vector field. The contact 1-form $\eta$ defines a $2n$-dimensional vector bundle $D$ over $M$, the fiber $D_{p}$ of $D$ is given by
\[ D_{p} = \ker \, \eta_{p} .\ \]
There is a decomposition of the tangent bundle $TM$,
\[ TM = D\oplus L_{\xi} \ \]
where $L$ is the trivial bundle generated by the Reeb vector $\xi$. The Riemannian metric $g$ and a tensor field $\Phi$ of type (1, 1) satisfy
\[ \Phi^{2} = - Id + \eta \otimes \xi \ \]
and
\[ g(\Phi X, \Phi Y) = g(X, Y) - \eta(X)\eta(Y)\ \]

 Now we exploit the transverse structure of Sasakian manifolds. On the subbundle $D$, it is naturally endowed with both a complex structure $\Phi|_{D}$ and a symplectic structure $d\eta$. Since both $g$ and $\xi$ are invariant under $\xi$, there is a well-defined K\"ahler structure $(g^{T}, \omega^{T}, J^{T})$ on the local leaf space of Reeb foliation $\mathcal{F_{\xi}}$. 
 We call this a transverse K\"ahler structure. Clearly
\[ g^{T}(X,Y) = \frac{1}{2}d\eta(X,\Phi Y).\ \]
\[ J^{T} = \Phi|_{D}\ \]
The metric $g^{T}$ is related to the Sasakian metric $g$ by
\[ g = g^{T} + \eta \otimes \eta.\ \]
The upper script $T$ is used to denote both the transverse geometry quantity, and the corresponding quantity on the bundle $D$. For the transverse metric $g^{T}$, there is a unique, torsion-free connection on the subbundle $D$, which is called the transverse Levi-Civita connection 
$$\nabla^{T}_{X}Y=
\begin{cases}
(\nabla_{X}Y)^{p}, & \text{$X \in D$}\\
[\xi,Y]^{p},& \text{$X = \xi$}
\end{cases}$$
where $Y$ is a section of $D$ and $X^{p}$ the projection of $X$ onto $D$, and it satisfies
\[ \nabla^{T}_{X}Y - \nabla^{T}_{Y}X - [X, Y]^{p} = 0, \ \]
 \[ Xg^{T}(Z, W) = g^{T}(\nabla^{T}_{X}Z, W) + g^{T}(Z, \nabla^{T}_{X}W). \ \]
 for any $X, Y \in TM$ and $Z,W \in D$. The transverse curvature relating with the above transverse connection is defined by
  \[ R^{T}(X, Y)Z =  \nabla^{T}_{X}\nabla^{T}_{Y}Z - \nabla^{T}_{Y}\nabla^{T}_{X}Z - \nabla^{T}_{[X,Y]}Z,\ \]
 where $X, Y \in TM$ and $Z \in D$. From the above transverse curvature operator we define the transverse Ricci curvature by
 \[Ric^{T}(X, Y) = \sum_{i}\langle R^{T}(X,e_{i})e_{i},Y  \rangle, \ \]
 Where \{$e_{i}$\} is an orthonormal basisi of $D$ and $X, Y \in D$. One can easily check that 
 \[Ric^{T}(X,Y) = Ric(X,Y) + 2g^{T}(X,Y) \ \] 
 for any $X,Y \in D.$ Let $\rho^{T}(X, Y) = Ric^{T}(\Phi X, Y)$. $\rho^{T}$ is called the transverse Ricci form, which is a representation of the first basic Chern class.
  
  We consider the complexified bundle $\mathcal{D}^{\mathbb{C}} = \mathcal{D} \otimes \mathbb{C}$. Using the structure $\Phi$ we decompose $\mathcal{D}^{\mathbb{C}}$ into two subbundles $\mathcal{D}^{(1, 0)}$ and $\mathcal{D}^{(0, 1)}$, where $\mathcal{D}^{(1, 0)} = \{ X \in \mathcal{D}^{\mathbb{C}}| \Phi X = \sqrt{-1}X \}$ and $\mathcal{D}^{(0, 1)} = \{ X \in \mathcal{D}^{\mathbb{C}}| \Phi X = - \sqrt{-1}X \}$.

 \begin{definition} Given a $\Phi$-invariant planes $\sigma$ in $D_{x}  \subseteq T_{x}M$, the transverse holomorphic sectional curvature $H^{T}(\sigma)$ is defined by
 \[ H^{T}(\sigma) = \langle R^{T}(X,JX)JX,X \rangle,\ \]
where $X$ is a unit vector in $\sigma$. It is easy to check that $\langle R^{T}(X,JX)JX,X \rangle$ only depend on $\sigma$. We say that the transverse holomorphic sectional curvature is negative (nonpositive) if $H^{T}(\sigma) < 0(\le 0)$, for any $\sigma$ in $D_{x}$ and any $x$ in $M$. The transverse holomorphic sectional curvature is said to be quasi-negative if $H^{T}(\sigma) \le 0$ and, moreover, there exists at least one point $p \in M$ such that  $H^{T}(\sigma) < 0$ for every $\sigma \in D_{p}$.
 \end{definition}
 
 Let $u$ be a unit vector on $D$, and a $\Phi$-invariant plane spanned by $u, \Phi u$ are denoted by $\sigma_{u}$. Setting $U = \frac{1}{2}(u - \sqrt{-1}\Phi u)$.Then, we have
 
 \[  \langle R^{T}(U,\bar{U})U,\bar{U} \rangle = \frac{1}{4}H^{T}(\sigma_{u}) \ \]
 By the above formula, we know that that the positivity of of transverse holomorphic sectional curvature is equivalent to $\langle R^{T}(U,\bar{U})U,\bar{U} \rangle > 0$, for all $U \in D^{(1,0)}.$
 
 On the Sasakian manifold $(M, \xi, \eta, \Phi, g)$, the basic Laplacian is defined by 
 \[ \Delta_{B}u = \frac{4n\sqrt{-1}\partial_{B}\bar{\partial}_{B}u\wedge(d\eta)^{n - 1}\wedge \eta}{(d\eta)^{n}\wedge\eta}. \ \]
 for any basic function $u$. It is well-known that the basic Laplacian is equal to the Riemannian Laplacian $\Delta_{g}$ on basic function, $i.e.$ $\Delta_{B}u = \Delta_{g}u$ for any basic function $u$.
 
  Fixing a transverse holomorphic structure on $\mathcal{F_{\xi}}$, there is also a notion of transverse cohomology on Sasakian manifolds. A $p$-form $\alpha$ on  $(M, \xi, \eta, \Phi, g)$ is called basic if $\iota_{\xi}\alpha = 0$, and $\mathcal{L}_{\xi}\alpha = 0.$ We let $\Lambda_{B}^{p}$ be the sheaf of basic $p$-forms, and $\Omega_{B}^{p} = \Gamma(M, \Lambda^{p}_{B})$ the global sections. It is easy to see that the de Rham differential $d$ preserves basic forms, and hence restricts to a well defined operator $d_{B}: \Lambda_{B}^{p} \to \Lambda_{B}^{p+1}.$ We thus get a complex
 \[ 0 \to C_{B}^{\infty}(M) \to \Omega_{B}^{1} \xrightarrow{d_{B}} \cdots  \xrightarrow{d_{B}}  \Omega_{B}^{2n} \xrightarrow{d_{B}} 0. \ \]
 whose cohomology groups, denoted by $H^{p}_{B}(M, \mathcal{F_{\xi}}),$ are the basic de Rham cohomology groups. Moreover, the transverse complex structure $\Phi$ allows us to decompose
 \[ \Lambda_{B}^{r} \mathop{\otimes} \mathbb{C} = \mathop{\oplus} \limits_{p + q = r}   \Lambda_{B}^{p,q}\ \]
 We can then decompose $d_{B} = \partial_{B} + \bar{\partial}_{B}$, where
 \[ \partial_{B}: \Lambda_{B}^{p,q} \to \Lambda_{B}^{p + 1,q}, \quad and \quad \bar{\partial}_{B}: \Lambda_{B}^{p,q} \to \Lambda_{B}^{p,q + 1}. \ \]
 So we can define the basic Dolbeault cohomology groups $H^{p, *}_{B}{(M, \mathcal{F_\xi})}.$ We also have the transverse Chern-Weil theory and can define the basic Chern classes $c^{B}_{k}(M, \mathcal{F_{\xi}})$. For detail, see \cite{BG3}.

  \begin{definition}
 Let $(M, \xi, \eta, \Phi, g)$ be a compact Sasakian manifold, a $(1, 1)$ basic class $\alpha \in H^{(1, 1)}_{B}(M, \mathcal{F_{\xi}})$ is called transverse nef. If every $\epsilon > 0$ there exist  basic $(1, 1)$ forms $\theta_{\epsilon} \in \alpha$ on $M$ such that  satisfies
 \begin{equation}
\theta_{\epsilon} + \epsilon d\eta \; is \;a \; basic \;transverse \;positive \;(1, 1)\;form. 
\end{equation}
\end{definition}
 
 On Sasakian manifolds, the $\partial\bar{\partial}$-lemma holds for basic forms. 
 \begin{proposition}
 Let $\theta$ and $\theta^{\prime}$ be two real closed basic forms of type $(1, 1)$ on a compact Sasakian manifold $(M, \xi, \eta, \Phi, g)$, If $[\theta]_{B} = [\theta^{\prime}]_{B} \in H^{(1, 1)}_{B}(M, \mathcal{F_{\xi}})$, then there is a basic real function $\phi$ such that
 \begin{equation}
 \theta - \theta^{\prime} = \sqrt{-1}\partial_{B}\bar{\partial}_{B}\phi.
 \end{equation}
 \end{proposition}

  We fix a canonical orientation $\eta\wedge(d \eta)^{n}$ and introduce  the concepts of transverse
positivity on Sasakian manifolds corresponding to the complex case. A basic $(p,p)$ form $\sigma(\ge 0)$ is said to be transverse positive if for any basic (1,0) forms $\gamma_{j}, 1 \le j \le p,$ then 
  \[ \sigma \wedge \sqrt{- 1}\gamma_{1}\wedge\bar{\gamma_{1}} \wedge  \cdots  \sqrt{- 1}\gamma_{n-p}\wedge\bar{\gamma}_{n-p} \wedge \eta \ \]
is a positive volume form.  
  
  Any transverse positive basic $(p, p)$ form $\sigma$ is real, i.e., $\sigma = \bar{\sigma}$. In particular, in the local coordinates, a real basic (1, 1) form
  \[ \sigma = \sqrt{- 1} \sigma_{i\bar{j}}dz_{i}\wedge dz_{\bar{j}}\ \]
  is transverse positive if and only if $(\sigma_{i\bar{j}})$ is a semipositive Hermitian matrix with $\xi(\sigma_{i\bar{j}}) = 0.$ We call a real basic (1,1) form $\sigma$ strictly transverse positive if the Hermitian matrix $(\sigma_{i\bar{j}})$ is positive definite. Given another real basic (1,1) form $\beta$. We can define 
  \[ \tr_{\sigma} \beta   :=   \frac{n\beta\wedge\sigma^{n - 1}}{\sigma^{n}}   = \sigma^{i\bar{j}}\beta_{i\bar{j}}\ \]
  where $(\sigma^{i\bar{j}})$ is the inverse matrix of $\sigma_{i\bar{j}}$. A basis (1, 1) class $\alpha \in H^{1, 1}_{B}{(M, \mathcal{F_\xi})}$ is called positive, if there exist a real  transverse positive basic (1, 1) form $\sigma \in \alpha$. We call a basic (1, 1) class $\alpha$ negative, if and only if $-\alpha$ is positive.

 \subsection{Foliated local coordinate and local computations}\label{Foliation}
 In this section, we first review local coordinates on a Sasakian manifold. In \cite{GKN}, it has been proven that every Sasakian manifold have a good local coordinates $(x, z^{1}, z^{2}, \cdots, z^{n})$ on a small neighborhood $U$ such that
 \[ \xi = \frac{\partial}{\partial x}, \ \]
 \[ \eta = dx - \sqrt{-1}(h_{j}dz^{j} - h_{\bar{j}}d\bar{z}^{j}), \ \]
 \[\Phi = \sqrt{-1}\{ X_{j} \otimes dz^{j} - \bar{X}_{j} \otimes d\bar{z}^{j} \}, \ \]
 \[g = \eta \otimes \eta + 2 h_{i\bar{j}}dz^{i}d\bar{z}^{j}, \ \]
 where $h: U \to \mathbb{R}$ is a local basic function, i.e., $\frac{\partial h}{\partial x} = 0, h_{i} = \frac{\partial h_{i}}{\partial z^{i}}, h_{i\bar{j}} = \frac{\partial^{2} h}{\partial z^{i} \partial z^{\bar{j}}},$ and $X_{j} = \frac{\partial}{\partial z^{j}} + \sqrt{-1}h_{j}\frac{\partial}{\partial x}, \bar{X}_{j} = \frac{\partial}{\partial \bar{z}^{j}} - \sqrt{-1}h_{\bar{j}}\frac{\partial}{\partial x}.$ we denote $2dz^{i}d\bar{z}^{j} = dz^{i} \otimes d\bar{z}^{j} + d\bar{z}^{j} \otimes dz^{i}.$ In such local coordinates, $D \otimes \mathbb{C}$ is spanned by $X_{i}$ and $\bar{X}_{i},$ it is clear that
 \[ \Phi X_{i} = \sqrt{-1} X_{i} , \quad \Phi \bar{X_{i}} = - \sqrt{-1} \bar{X_{i}} ,\ \]
 \[ [X_{i},X_{j}] = [\bar{X_{i}},\bar{X_{j}}] = [\xi,\bar{X_{i}}] = [\xi,\bar{X_{j}}] = 0,\ \]
 \[ [X_{i},\bar{X_{j}}] = - 2 \sqrt{-1}h_{i\bar{j}}\xi. \ \]
 Obviously, $\{\eta, dz^{i}, dz^{\bar{j}} \}$ is the dual basis of $\{ \frac{\partial}{\partial x}, X_{i}, \bar{X_{j}} \},$ and
 \[ \omega^{T} = \frac{1}{2}d\eta = \sqrt{-1}h_{i\bar{j}}dz^{i}d\bar{z}^{j}, \ \]
 the transverse metric 
 \[ g^{T} = 2g^{T}_{i\bar{j}}dz^{i}d\bar{z}^{j} = 2h_{i\bar{j}}dz^{i}d\bar{z}^{j},\ \]
 where $g^{T}_{i\bar{j}} = g^{T}(X_{i},\bar{X_{j}}) = h_{i\bar{j}}.$ From the formula above, we know that $\nabla^{T}_{\frac{\partial}{\partial x}}X_{i} = \nabla^{T}_{\frac{\partial}{\partial x}}\bar{X}_{j} = 0.$ Define $\Gamma^{A}_{BC}$ by
 \[ \nabla^{T}_{X_{B}}X_{C} = \Gamma^{A}_{BC} X_{A}. \ \]
 for $A,B,C = 1,2,\cdots,n,\bar{1},\bar{2},\dots,\bar{n},$ where $X_{\bar{j}} = \bar{X_{j}}.$ It is easy to check that only $\Gamma^{i}_{jk}$ and $\Gamma^{\bar{i}}_{\bar{j}\bar{k}}$ may not vanish as in the K\"ahler case. Moreover,
 \[\Gamma^{i}_{jk} = \Gamma^{i}_{kj} = h^{i\bar{l}}\frac{\partial h_{j\bar{l}}}{\partial z^{k}}.\ \]
This local coordinates are also called by a normal coordinates on Sasakian manifold. We can cover $M$ by finite foliated local coordinate charts $\{U_{\alpha}\}$ which is diffeomorphism to $(-\epsilon_{0}, \epsilon_{0}) \times B_{2}(0)$ with $\epsilon_{0} > 0,$ where $B_{2}(0)$ is the ball in $\mathbb{C}^{n}$ centered at origin with radius 2, and on $B_{2}(0)$, there holds 
\begin{equation}
C^{-1}\delta_{ij} \le d\eta \le C\delta_{ij}
\end{equation} 
for a uinform constant $C$. Moreover, $\{\frac{1}{2}U_{\alpha}\}$ is diffeomorphism to $(-\frac{1}{2}\epsilon_{0}, \frac{1}{2}\epsilon_{0}) \times B_{1}(0)$ still cover $M$. We have basic Sobolev space
\begin{equation}
C^{k,\alpha}_{B}(M) = \{ u \; | \;  \xi u = 0, ||u||_{C^{k ,\alpha}(M)} < \infty  \},
\end{equation}
where we ues the following notation : in finite foliated local coordinate charts $\{U_{i}\}_{i = 1}^{i = m}$
\[ ||u||_{C^{k ,\alpha}(M)} :=  \sum_{1\le i \le m}(\sum_{0\le j \le k} \sup_{U_{i}}|D^{j}u| + \sup_{x, y \in U_{i}, x \ne y}\frac{|D^{k}u(x) - D^{k}u(y)|^{\alpha}}{|x - y|}) \ \]

\begin{remark}For a fixed point $P \in M$, we can always choose the above coordinates $(x, z^{1}, z^{2}, \cdots, z^{n})$ centered at $P$ satisfying additionally that $\{(\frac{\partial}{\partial z^{i}}|_{P}) \} \in D^{(1,0)}$ or equivalently $h_{i}(P) = 0$ for all $i.$ Furthermore, in the same way as that in K\"ahler case, one can choose a normal coordinates $(x, z^{1}, z^{2}, \cdots, z^{n})$, such that $h_{i} = 0, h_{i\bar{j}}(P) = \delta_{ij},$ and $dh_{i\bar{j}}|_{P}=0,$ i.e., $\Gamma^{i}_{jk}|_{P} = 0$ for all $i, j, k$. 
\end{remark} 
 
 One can easily check that the transverse Ricci curvature can be expressed by 
 \[ R^{T}_{i\bar{j}} = - \frac{\partial^{2}}{\partial z^{i} \partial \bar{z}^{j}}\log \det(g^{T}_{m\bar{n}}). \ \]
 and $\rho^{T} = \sqrt{-1}R^{T}_{i\bar{j}}dz^{i}\wedge d\bar{z}^{j}.$
 
 Suppose that $(\xi, \eta, \Phi, g)$ defines a Sasakian structure on $M.$ Let $\varphi$ be a basic function satisfying
 \[\eta_{\varphi} \wedge (d \eta_{\varphi})^{n}  \ne 0. \ \]
  set 
 \[\eta_{\varphi} = \eta + \frac{\sqrt{-1}}{2}(\bar{\partial}_{B} - \partial_{B}) \varphi. \ \]
 \[\Phi_{\varphi} = \Phi - \xi \otimes(d^{c}_{B}\varphi)\circ \Phi,\;\; g_{\varphi} = \frac{1}{2}d\eta_{\varphi}\circ(Id\otimes\Phi_{\varphi}) + \eta_{\varphi} \otimes \eta_{\varphi}.\ \]
 It is clear that 
 \[ d\eta_{\varphi} = d\eta + \sqrt{-1}\partial_{B}\bar{\partial}_{B}\varphi.\ \]
 and $(\xi, \eta_{\varphi}, \Phi_{\varphi}, g_{\varphi})$ is also a Sasakian structure on $M$. Furthermore, $(\xi, \eta_{\varphi}, \Phi_{\varphi}, g_{\varphi})$ and $(\xi, \eta, \Phi, g)$ have the same Reeb field and the same transversely holomorphic structure, $i.e.$ the Sasakian structures compatible with $(\xi, \eta, \Phi, g)$. By El-Kacimi's \cite{KA} generalization of Yau's estimates for transverse Monge-Amp\`ere equations. For any $\alpha \in c^{B}_{1}(M, \mathcal{F_{\xi}})$, there exists a compatible Sasakian structure $(\xi, \eta_{\varphi}, \Phi_{\varphi}, g_{\varphi})$ such that the transverse Ricci form satisfying
 \[ \rho^{T}_{\varphi} = \alpha. \ \]
 
\section{Proof of Theorem 1.1} 

In this section, We will follow Wu and Yau's method \cite{WY1, WY2} to prove that a compact Sasakian manifold with negative transverse holomorphic sectional curvature has the negative first basic Chern class. We will prove Theorem \ref{Th1}  by solve a family of transverse Monge-Amp\`ere equations parameterized by $t$ and consider the limit of the family of transverse basic positive (1, 1)-form 
\begin{equation}
td\eta - \rho^{T}+ \sqrt{-1} \partial_{B}\bar{\partial}_{B} u_{t}.
\end{equation}

 By the transverse Aubin-Yau theorem \cite{FZ, KA, SWZ}, we have

\begin{theorem}
Let $(M, \xi, \eta, \Phi, g)$ be a compact Sasakian manifold with $dim_{\mathbb{R}}M$ = 2n + 1. $\frac{1}{2}d \eta$ is the transverse K\"ahler form corresponding to the transverse metric $g^{T}$ defined by $g^{T}(X, Y) = \frac{1}{2} d \eta(X, \Phi Y)$. Let $\sigma$ be a strictly transverse positive basic real  $(1,1)$-form on $M$. There exists a unique basic function $u \in C^{\infty}_{B}(M,\mathbb{R})$ solves the equation
\begin{equation}
(\sigma + \sqrt{-1}\partial_{B}\bar{\partial}_{B}u)^{n}\wedge \eta = e^{u}(d\eta)^{n}\wedge \eta.
\end{equation}
where $\sigma + \sqrt{-1}\partial_{B}\bar{\partial}_{B}u$ is a strictly transverse positive basic real $(1, 1)$ form.
\end{theorem}

  \begin{proof}[Proof of Theorem 1.1]
For $t  > 0$, We consider a family of transverse Monge-Amp\`ere equations
\begin{equation}\label{MA}
 (td\eta - \rho^{T}+ \sqrt{-1} \partial_{B}\bar{\partial}_{B} u_{t})^{n}\wedge \eta = e^{u_{t}}(d\eta)^{n}\wedge \eta.
\end{equation}
Since $d\eta$ is a transverse K\"ahler form and $M$ is compact, there exists a sufficiently large constant $t_{1} > 0$ such that 
$t_{1}d\eta - \rho^{T}+ \sqrt{-1} \partial_{B}\bar{\partial}_{B} u_{t_{1}} > 0.$ Fix a nonnegative integer $k$ and $0 < \alpha < 1.$ 
We use the notation that 
\begin{equation}
C^{k + 2,\alpha}_{B}(M) = \{ u \in C^{k + 2,\alpha}(M)| \;  \xi u = 0 \}.
\end{equation}
Define
\begin{equation}
I = \{ t \in [0,t_{1}]| \; there \; is \;   a  \;  solution  \;  u_{t}  \;  \in   \;  C^{k + 2,\alpha}_{B}(M)  \;satisfying  \; (\ref{MA}) \}.
\end{equation}
First, by Theorem 1.1, we know $I \ne \emptyset,$ since $t_{1} \in I.$

That $I$ is open in $[0, t_{1}]$ follows from the implicit function theorem. Let $t_{0} \in I$ with corresponding function $u_{t_{0}} \in C^{k + 2,\alpha}_{B}(M)$. Then, there exists a small neighborhood $J$ of $t_{0}$ in $[0, t_{1}]$ and a small neighborhood $U$ of $u_{t_{0}}$ in $C^{k + 2,\alpha}_{B}(M)$ such that
\begin{equation}
\sigma_{t} = td\eta - \rho^{T} + \sqrt{-1}\partial_{B}\bar{\partial}_{B}u > 0.
\end{equation}
for all $t \in J$ and $u \in U.$
Define a map $\Phi: J \times U \to C^{k,\alpha}_{B}(M)$.
\begin{equation}
\Phi(t,u) = \log\frac{(td\eta - \rho^{T} + \sqrt{-1}\partial_{B}\bar{\partial}_{B}u)^{n}\wedge \eta}{(d\eta)^{n}\wedge \eta} - u.
\end{equation}
with $\Phi(t_{0},u_{t_{0}}) = 0$. It suffices to prove the invertibility of the linearization
\begin{equation}
(D\Phi)_{u}(t_0,u_{t_{0}}):  C^{k + 2,\alpha}_{B}(M) \to C^{k,\alpha}_{B}(M)
\end{equation}
where the linearization is
\begin{equation}
(D\Phi)_{u}(t_0,u_{t_{0}})h = \frac{d}{ds}\Phi(t_{0},u_{t_{0}}+sh)|_{s=0} = \sigma^{i\bar{j}}_{t_{0}}h_{i\bar{j}} - h.
\end{equation}
By maximum principle, we know that the linearization is injective. From the implicit function theorem, we know $I$ is open. The estimate (\ref{3.20}) shows the closeness of $I$. In particular, $0 \in I$ with corresponding $u_{0} \in C^{\infty}_{B}(M)$. This gives us the desired strictly transverse positive basic (1, 1) form  $ -\rho^{T} + \sqrt{-1}\partial_{B}\bar{\partial}_{B}u_{0} \in -c_{1}^{B}(M, \mathcal{F_{\xi}}).$ By the transverse Calabi-Yau theorem, $M$ has a Sasakian structure $(\xi, \eta^{'}, 
\Phi^{'}, g^{'})$ 
\[ \rho^{T}(g^{'}) = \rho^{T} - \sqrt{-1}\partial_{B}\bar{\partial}_{B}u_{0} \ \]
which is transverse negative and compatible with $(\xi, \eta, \Phi, g)$.
\end{proof}

\begin{proposition}\label{C2}
Let $(M, \xi, \eta, \Phi, g)$ be a compact Sasakian manifold with the upper bounded transverse holomorphic sectional curvature by $-\kappa$ $(\kappa > 0)$, assuming that $\sigma(t)$ is a solution of (\ref{MA}), where $t \in I$. We have the uniformly second estimate
\[ \tr_{\sigma(t)}d\eta \le \frac{2n\kappa}{n + 1}, \ \]
for all $t \in I$.
\end{proposition}

\begin{proof}
Choose a normal coordinate system $(x, z^{1}, z^{2}, \cdots, z^{n})$ about Sasakian metric $g$ near a point $P$ of $M$ such that $\{ \sigma_{i\bar{j}}\}$ is diagonal, we have
\begin{equation}
\begin{split}
 \sigma^{k\bar{l}}(t)\partial_{k}\partial_{\bar{l}}\tr_{\sigma(t)}d\eta &= \sigma^{i\bar{j}}(t)\partial_{i}\partial_{\bar{j}}(\sigma^{k\bar{l}}(t)g^{T}_{k\bar{l}}) \\&= \sigma^{i\bar{j}}(t)\sigma^{k\bar{l}}(t)\partial_{i}\partial_{\bar{j}}g^{T}_{k\bar{l}} + \sigma^{i\bar{j}}(t)\partial_{i}\partial_{\bar{j}}\sigma^{k\bar{l}}(t)g^{T}_{k\bar{l}} 
\end{split}
\end{equation}
where
\[ R^{T}_{i\bar{j}k\bar{l}} = - \partial_{i}\partial_{\bar{j}}g^{T}_{k\bar{l}}.\ \]
By Royden's lemma \cite{roy} , we have
\begin{equation}
\begin{split}
\sigma^{i\bar{j}}(t)\sigma^{k\bar{l}}(t)R^{T}_{i\bar{j}k\bar{l}} \le -\frac{n + 1}{2n}\kappa (\tr_{\sigma(t)}d\eta)^{2}
\end{split}
\end{equation}

By some calculations, we have

\begin{equation}\label{3.8}
\begin{split}
 \sigma^{i\bar{j}}(t)\partial_{i}\partial_{\bar{j}}\sigma^{k\bar{l}}(t)g^{T}_{k\bar{l}} &= \sigma^{i\bar{i}}(t)(- \partial_{i}\partial_{\bar{i}}\sigma_{k\bar{k}}(t)\sigma^{k\bar{k}}(t) \sigma^{k\bar{k}}(t) \\& + \sigma^{k\bar{k}}(t) \sigma^{k\bar{k}}(t) \sigma^{q\bar{q}}(t)  \partial_{i}\sigma_{q\bar{k}}(t)\partial_{\bar{i}}\sigma_{k\bar{q}}(t) \\&
  + \sigma^{k\bar{k}}(t)\sigma^{k\bar{k}}(t)\sigma^{p\bar{p}}(t) \partial_{\bar{i}}\sigma_{p\bar{k}}(t)\partial_{i}\sigma_{k\bar{p}}(t))
\end{split}
\end{equation}

\begin{equation}\label{3.9}
\begin{split}
 -\partial_{i}\partial_{\bar{j}}\log \det \sigma(t)& = - \sigma^{k\bar{k}}(t)\partial_{i}\partial_{\bar{j}}\sigma_{k\bar{k}}(t) + \sigma^{k\bar{k}}(t)\partial_{\bar{j}}\sigma_{p\bar{k}}(t)\sigma^{p\bar{p}}(t)\partial_{i}\sigma_{k\bar{p}}(t) \\&= -\sigma^{k\bar{k}}(t)(\partial_{k}\partial_{\bar{k}}\sigma_{i\bar{j}}(t) - \partial_{\bar{j}}\sigma_{p\bar{k}}(t)\sigma^{p\bar{p}}(t)\partial_{i}\sigma_{k\bar{p}}(t))
\end{split}
\end{equation}
By (\ref{3.8}) and (\ref{3.9}), we have 

\begin{equation}
\begin{split}
 \sigma^{i\bar{j}}(t)\partial_{i}\partial_{\bar{j}}\sigma^{k\bar{l}}(t)g^{T}_{k\bar{l}}  &= - \partial_{k}\partial_{\bar{k}}\log \det\sigma(t) \sigma^{k\bar{k}}(t)\sigma^{k\bar{k}}(t) + \\& \sigma^{i\bar{i}}(t) \sigma^{k\bar{k}}(t) \sigma^{q\bar{q}}(t) \sigma^{k\bar{k}}(t)\partial_{i}\sigma_{q\bar{k}}(t)\partial_{\bar{i}}\sigma_{k\bar{q}}(t)
\end{split}
\end{equation}
We  know that (\ref{MA}) can be written locally
\begin{equation}
\log\frac{\det(\sigma_{j\bar{k}}(t))}{\det g^{T}_{j\bar{k}}} = u(t)
\end{equation}
We now apply the $\partial\bar{\partial}$-operator to the above equation to get
\begin{equation}
 - \sqrt{-1}\partial\bar{\partial}\log \det \sigma(t) = - \sigma(t) + td\eta
\end{equation}
We apply the Cauchy-Schwarz inequality (\cite{WYZ}, p. 372) to obtain
\begin{equation}
\begin{split}
\sigma^{i\bar{i}}(t) \sigma^{k\bar{k}}(t) \sigma^{q\bar{q}}(t) \sigma^{i\bar{i}}(t)\partial_{k}\sigma_{q\bar{i}}(t)\partial_{\bar{k}}\sigma_{i\bar{q}}(t) &\ge \sum_{k, i}\sigma^{k\bar{k}}(t)\sigma^{i\bar{i}}(t)^{3} |\partial_{k}\sigma_{i\bar{i}}|^{2} \\&\ge \frac{1}{tr_{\sigma(t)}d\eta}\sigma^{k\bar{k}}(t)\partial_{k}tr_{\sigma(t)}d\eta\partial_{\bar{k}}tr_{\sigma(t)}d\eta
\end{split}
\end{equation}
by  we have

\begin{equation}
\begin{split}
\sigma^{i\bar{j}}(t)\partial_{i}\partial_{\bar{j}}\sigma^{k\bar{l}}(t)g^{T}_{k\bar{l}}  &= (-\sigma_{i\bar{i}}(t)  + tg^{T}_{i\bar{i}})\sigma^{i\bar{i}}(t)\sigma^{i\bar{i}}(t) +\\& \sigma^{i\bar{i}}(t) \sigma^{k\bar{k}}(t) \sigma^{q\bar{q}}(t) \sigma^{i\bar{i}}(t)\partial_{k}\sigma_{q\bar{i}}(t)\partial_{\bar{k}}\sigma_{i\bar{q}}(t) \\& \ge -\tr_{\sigma(t)}d\eta + \frac{t}{n}(\tr_{\sigma(t)}{d\eta})^{2} \\& + \sigma^{i\bar{i}}(t) \sigma^{k\bar{k}}(t) \sigma^{q\bar{q}}(t) \sigma^{i\bar{i}}(t)\partial_{k}\sigma_{q\bar{i}}(t)\partial_{\bar{k}}\sigma_{i\bar{q}}(t)
\end{split}
\end{equation}

\begin{equation}\label{key}
\begin{split}
\sigma^{k\bar{l}}(t)\partial_{k}\partial_{\bar{l}} \log \tr_{\sigma(t)}d\eta &= \frac{1}{\tr_{\sigma(t)}d\eta}(\sigma^{k\bar{l}}(t)\partial_{k}\partial_{\bar{l}}\tr_{\sigma(t)}d\eta -  \frac{1}{\tr_{\sigma(t)}d\eta}\partial_{k}\tr_{\sigma(t)}d\eta \partial_{\bar{l}}\tr_{\sigma(t)}d\eta \sigma^{k\bar{l}}(t))  \\& \ge \frac{1}{\tr_{\sigma(t)}d\eta} (\frac{n + 1}{2n}\kappa (\tr_{\sigma(t)}d\eta)^{2} - \tr_{\sigma(t)}d\eta + \frac{t}{n}(\tr_{\sigma(t)}d\eta)^{2})  \\& \ge -1 + \frac{n + 1}{2n}\kappa(\tr_{\sigma(t)}d\eta)
\end{split}
\end{equation}

By maximum principle, we have 
\begin{equation}\label{3.18}
\tr_{\sigma(t)}d\eta \le \frac{2n}{(n + 1)\kappa}.
\end{equation}
for all $t \in I$.
\end{proof}

\begin{lemma}
Let $(M, \xi, \eta, \Phi, g)$ be a compact Sasakian manifold with the upper bounded transverse holomorphic sectional curvature by $-\kappa$ $(\kappa > 0)$, assuming that $\sigma(t)$ is a solution of (\ref{MA}), where $t \in I$. There exist a uniform constant $C$ which is independent of $\epsilon$ such that
\begin{equation}\label{3.19}
\sup \limits_{M}|u_{t}| \le C.
\end{equation} 
\end{lemma}
\begin{proof}
Choose a point $x_{t} \in M$ where $u(t)$ achieves its maximum, we know that $td\eta - \rho^{T}$ is a strictly transverse positive at $x_{t}$ and 
\begin{equation}
 e^{\sup_{M}u_{t}} = e^{u_{t}(x_{t})} \le \frac{(td\eta - \rho^{T} )^{n}\wedge \eta}{(d\eta)^{n}\wedge \eta}(x_{t}) \le C.
\end{equation}
 so we have 
 \begin{equation}\label{3.21}
 \sup \limits_{M} \frac{\sigma(t)^{n}\wedge\eta}{(d\eta)^{n}\wedge \eta} \le C.
 \end{equation}
 Combining (\ref{3.18}) and (\ref{3.21})  and the elementary inequality
 \begin{equation}
 \tr_{d\eta}\sigma(t) \le \tr_{\sigma(t)}d\eta \frac{\sigma(t)^{n}}{(d\eta)^{n}\wedge \eta}
 \end{equation}
 we have 
 \begin{equation}\label{3.23}
 \tr_{d\eta}\sigma(t) \le C,
 \end{equation}
 and (\ref{3.18}) and (\ref{3.23}) together give

 \begin{equation}\label{3.24}
\frac{1}{C}d\eta \le \sigma(t) \le Cd\eta.
\end{equation}
Combining (\ref{MA}) and (\ref{3.24}) we have the lower uniform bound estimate about $u(t)$
\begin{equation}
 \inf \limits_{M} u(t) \ge -C.
\end{equation}
\end{proof}
 The $C^{2,\alpha}$-estimate about (\ref{MA}) follows Blocki's theorem.
\begin{theorem}[\cite{Bl}, Theorem 3.1]\label{th3.3}
Let $u$ be a $C^{4}$-psh funtion in an open $\Omega \subseteq \mathbb{C}^{n}.$ Assume that foe some positive $K_{0}, K_{1}, K_{2}, b, B_{0}$ and $B_{1}$, we have 
\[ |u| \le K_{0}, |Du| \le K_{1}, \Delta u \le K_{2}\ \]
and 
\[b\le \Phi \le B_{0}, |D\Phi^{\frac{1}{n}}| \le B_{1}\ \]
in $\Omega,$ where $\Phi = \det(u_{i\bar{j}})$. Let $\Omega^{'}\Subset\Omega$. Then there exist $\alpha \in (0, 1)$ depending, besides these constants, on $dis(\Omega^{'},\partial\Omega)$ such that 
\[ ||D^{2}u ||_{C^{\alpha}(\Omega^{'})}\le C .\]
\end{theorem}

In the foliated local coordinate patch $(-\epsilon_{0}, \epsilon_{0}) \times B_{2}(0)$, we work in $B_{2}(0).$ Combining (\ref{3.19}) and (\ref{3.24}), the $C^{2, \alpha}$ estimate follows from Theorem \ref{th3.3}. The transverse elliptic Schauder estimates give the higher order estimates.
\begin{equation}\label{3.20}
||u(t)||_{C^{k}_{B}(M)} \le C_{k}
\end{equation}
where $C_{k}$ is independent of $\epsilon$, for all $k \ge 0$ and $t \in I$.

\section{Proof of Theorem 1.2} 
\begin{proof}[Proof of Theorem 1.2]
If $-c^{B}_{1}(M, \mathcal{F_\xi})$ is not transverse nef, then 
\begin{equation}\label{assu}
\epsilon_{0} = \inf\{\epsilon > 0 | \exists \; u_{\epsilon} \in C^{\infty}_{B}(M) \;s.t. \;\; \epsilon d\eta - \rho^{T}+ \sqrt{-1} \partial_{B}\bar{\partial}_{B}u_{\epsilon} > 0 \}
\end{equation}
is positive.
For $\epsilon  > 0$, We consider a family of transverse Monge-Amp\`ere equations
\begin{equation}\label{MA2}
 ((\epsilon + \epsilon_{0})d\eta - \rho^{T}+ \sqrt{-1} \partial_{B}\bar{\partial}_{B} u_{\epsilon})^{n}\wedge \eta = e^{u_{\epsilon}}(d\eta)^{n}\wedge \eta.
\end{equation}

Denoting $ \sigma(\epsilon) = (\epsilon + \epsilon_{0})d\eta - \rho^{T}+ \sqrt{-1} \partial_{B}\bar{\partial}_{B} u_{\epsilon}$. We have uniformly $C^{2}$-estimate as Proposition \ref{C2}
\begin{equation}\label{C22}
\begin{split}
\sigma^{k\bar{l}}(\epsilon)\partial_{k}\partial_{\bar{l}} \log \tr_{\sigma(\epsilon)}d\eta &\ge \frac{1}{\tr_{\sigma(\epsilon)}d\eta} (\frac{n + 1}{2n}\kappa (\tr_{\sigma(\epsilon)}d\eta)^{2} - \tr_{\sigma(\epsilon)}d\eta + \frac{\epsilon_{0}}{n}(\tr_{\sigma(\epsilon)}d\eta)^{2})  \\& \ge -1 + \frac{\epsilon_{0}}{n}\tr_{\sigma(\epsilon)}d\eta,
\end{split}
\end{equation}
so we will have 
\begin{equation}\label{hg2}
 \frac{1}{C}d\eta \le \sigma(\epsilon) \le Cd\eta.
\end{equation}
 where $C$ is independent of $\epsilon$. The higher order estimates is the same argument as Theorem \ref{Th1}.
 \begin{equation}
 \|u_\epsilon\|_{C^{k}(M)} \le C_{k}.
 \end{equation}
  where $C_{k}$ is independent of $\epsilon$, for all $k \ge 0.$ By Ascoli-Arzel\`a theorem and a diagonal argument with (\ref{C22}) and (\ref{hg2}), we obtain that there exists a sequence $\epsilon _{i}\to 0$ such that $u_{\epsilon _{i}}$ converges smoothly to a basic function $u_{0}$ and $\sigma_{\epsilon_{i}}$ converges smoothly to a strictly positive basic $(1, 1)$ form  
     \begin{equation}
  \sigma_{0} = \epsilon_{0}d\eta - \rho^{T} + \sqrt{-1} \partial_{B}\bar{\partial}_{B}u_{0}. 
  \end{equation}
  
  Since $\sigma_{0}$ is a strictly transverse positive basis (1, 1) form and $M$ is compact, we know  
     \begin{equation}
  \sigma_{0} = (\epsilon_{0} - \epsilon)d\eta_{B} - \rho^{T} + \sqrt{-1} \partial_{B}\bar{\partial}_{B}u_{0}. 
  \end{equation}
    is also a strictly transverse positive basis (1, 1) form, when $\epsilon$ is enough small. This is contradictory to the assumption in (\ref{assu}). So $\epsilon_{0}$ is equal to 0, then $-c^{B}_{1}(M, \mathcal{F_\xi})$ is transverse nef.
\end{proof}

\section{Proof of Theorem 1.3} 

In this section, we will prove Theorem \ref{Th3} following the method in \cite{WY2}. Zhang (\cite{Z1}) generalized the exponential estimates of plurisubharmonic functions in K\"ahler geometry to the Sasakian situation.

\begin{lemma}[\cite{Z1}, Proposition 3.3]
Let $\sigma$ be a basic (1, 1) form on a compact Sasakian manifold  $(M, \xi, \eta, \Phi, g)$, and $\mathcal{H}(M, \xi, \eta, \Phi, g, \sigma) = \{ u \in C^{\infty}_{B}(M)|\sigma +
− \sqrt{-1} \partial_{B}\bar{\partial}_{B}u \ge 0\}.$ Then, there exist two positive constants $\alpha$ and $C$, where C depends only on $\alpha$ and the geometry of $(M, \xi, \eta, \Phi, g)$ and $\sigma$, such that
\begin{equation}
\int_{M}e^{- \alpha(u - \max_{M}u)}(d\eta)^{n}\wedge \eta \le C.
\end{equation}
\end{lemma}

The next lemma is from \cite{WY2}. 
\begin{lemma}
Let $u$ be a negative basic $C^{2}$ function on a compact Sasakian manifold $(M, \xi, \eta, \Phi, g)$. Suppose the basic Lapalicain $\Delta_{B}u \ge - v$ for some continuous basic function $v$ on $M$. then 
\[\int_{M} |\nabla \log(-u)|^{2} (d\eta)^{n}\wedge\eta \le \frac{1}{\min_{M}{(-u)}}\int_{M}|v|(d\eta)^{n}\wedge\eta. \ \]
where $\nabla$ is the Levi-Civita connection corresponding to the Riemannian metric $g$.
\end{lemma}
Because the basic Laplacian is equal to the Riemannian Laplacian $\Delta_{g}$ on basic functions, so the proof is the same as the argument in Lemma 4 (\cite{WY2}).

\begin{lemma}\label{Lem 5.3}
Let $\sigma$ be a basic (1, 1) form on a compact Sasakian manifold  $(M, \xi, \eta, \Phi, g)$, and $\mathcal{H}(M, \xi, \eta, \Phi, g, \sigma) = \{ u \in C^{\infty}_{B}(M)|\sigma +
− \sqrt{-1} \partial_{B}\bar{\partial}_{B}u \ge 0\}.$ Then $v \equiv u - \max_{M}u - 1$ satisfies 
\begin{equation}
\int_{M}|\log(- v)|^{2} (d\eta)^{n}\wedge \eta + \int_{M}|\nabla \log(- v)|^{2}(d\eta)^{n}\wedge \eta \le C.
\end{equation}
where $C > 0$ is a constant depending only on the geometry of $(M, \xi, \eta, \Phi, g)$ and $\sigma$. So any sequence 
\[\log(1 + \max_{M}u_{k} - u_{k}) \ \]
is relatively compact in $L^{2}_{B}(M)$, where $u_{k} \in (M, \xi, \eta, \Phi, g)$.
\end{lemma}

\begin{proof}[Proof of Theorem 1.3]
By Theorem \ref{Th2}, we know that the transverse Monge-Amp\`ere equations
\begin{equation}\label{MA3}
\begin{split}
 (td\eta - \rho^{T}&+ \sqrt{-1} \partial_{B}\bar{\partial}_{B} u_{t})^{n}\wedge \eta = e^{u_{t}}(d\eta)^{n}\wedge \eta.\\
\\& td\eta - \rho^{T}+ \sqrt{-1} \partial_{B}\bar{\partial}_{B} u_{t} > 0\\
 \end{split}
\end{equation}
have solutions for any $t > 0.$ 
We denote $\sigma(t) = td\eta - \rho^{T}+ \sqrt{-1} \partial_{B}\bar{\partial}_{B} u_{t}$. Clearly,
\begin{equation}\label{5.11}
 \int_{M} (- c_{1}^{B}(M, \mathcal{F_{\xi}}))^{n}\wedge \eta = \lim \limits_{ t \to 0} \int_{M} \sigma(t)^{n}\wedge \eta = \lim \limits_{ t \to 0} \int_{M} e^{u_{t}}(d\eta)^{n}\wedge \eta.
\end{equation}
By inequality (\ref{key}) and Cauchy-Schwarz inequality we have 
\begin{equation}
\tr_{\sigma(t)}\sqrt{-1}\partial_{B}\bar{\partial}_{B}\log \tr_{\sigma(t)}d \eta \ge \frac{(n + 1)\kappa}{2}\exp(- \frac{\max_{M}u_{t}}{n}) - 1.
\end{equation}
Integrating inequality (\ref{key}) w.r.t volume element $\sigma(t)^{n}\wedge \eta$ we have
\begin{equation}\label{5.5}
\begin{split}
\exp(- \frac{\max_{M}u_{t}}{n}) &\le \frac{\int_{M}\sigma(t)^{n}\wedge \eta}{\frac{n + 1}{2}\int_{M}\kappa\sigma(t)^{n}\wedge \eta}
\\& \le  \frac{\int_{M}\exp(u_{t} - \max_{M}u_{t} - 1)(d\eta)^{n}\wedge \eta}{\frac{n + 1}{2}\int_{M}\kappa\exp(u_{t} - \max_{M}u_{t} - 1)(d\eta)^{n}\wedge \eta}
\end{split}
\end{equation} 
Since $t_{1}d\eta - \rho^{T}+ \sqrt{-1} \partial_{B}\bar{\partial}_{B} u_{t} > \sigma_{t} \ge 0$ for any $0 < t \le t_{1},$ by Lemma \ref{Lem 5.3}, we know the set 
\[ {\log(1 + \max_{M}u_{t} - u_{t})}; 0 < t \le t_{1} \ \]
is uniformly bounded in $W^{1, 2}_{B}$ and $\log(1 + \max_{M}u_{t_{i}} - u_{t_{i}})$ converges to $w$ almost  everywhere on $M.$ By Lebesgue dominated convergence theorem the right side of (\ref{5.5}) converges to $C > 0.$ So we have the lower bounder of $\max_{M}u_{t_{i}}$, the upper bound get from (\ref{MA3}) by maximum principle. Up to a sequence we can assume $u_{t_{i}}$ converges to $-e^{w} + c.$ Plugging these back to (\ref{5.11}), we prove Theorem 1.3.
\end{proof}

\section{Proof of Theorem 1.4} 
 In this section we will prove the Miyaoka-Yau inequality for basic Chern numbers holds on compact Sasakian manifolds with nonpositive transverse holomorphic sectional curvature. 
\begin{proof}[Proof of Theorem 1.4]
Let $(M, \xi, \eta, \Phi, g)$ be a $2n + 1$ dimensional compact Sasakian manifold with nonpositive transverse holomorphic sectional curvature. We can cover $M$ by finite foliated local coordinate charts $\{U_{\alpha}\}$. Assume $\sigma$ is a transverse strictly positive basic (1, 1) form. Then we have the induced metric on the contact bundle $D$.
\[ h^{T}(X, Y) = \sigma(X, \Phi Y). \ \]
Where $X, Y \in \Gamma(D).$ So we have a Riemann metric $g$ on $M$
\[ g = h^{T} + \eta \otimes \eta. \ \]

We can define a transverse connection $\nabla^{T}$
$$\nabla^{T}_{X}Y=
\begin{cases}
(\nabla_{X}Y)^{p}, & \text{$X \in D$}\\
[\xi,Y]^{p},& \text{$X = \xi$}
\end{cases}$$
where $Y$ is a section of $D$ and $X^{p}$ the projection of $X$ onto $D$,
Let $X_{i}$ be local foliate transverse frame on the contact bundle $D$, by argument of section (\ref{Foliation}), we have
\[ \nabla^{T}_{X_{i}}X_{j} = \Gamma^{k}_{ij} X_{k}, \ \]
 \[\Gamma^{i}_{jk} = \Gamma^{i}_{kj} = h^{i\bar{l}}\frac{\partial h_{j\bar{l}}}{\partial z^{k}},\ \]
 Where $X_{i},X_{j} \in D^{(1, 0)}$ and $X \in D$. By some computation, we know the transverse curvature of $\nabla^{T}$ is 
 \begin{equation}
 R^{T}_{i\bar{j}k\bar{l}} = - \partial_{i}\partial_{\bar{j}}h^{T}_{k\bar{l}} + (h^{T})^{p\bar{q}}\partial_{i}h_{k\bar{q}}^{T}\partial_{\bar{j}}h_{p\bar{l}}^{T} .
 \end{equation}
 From the proof of Theorem 1.1, we have a family of transverse strictly positive basic (1, 1) form $\sigma(t)$, $t > 0.$
 By direct calculation, we have
 \begin{equation}\label{6.2}
 \begin{split}
 &(2\pi)^{2}\int_{M}(2c_{2}^{B}(M, \mathcal{F_{\xi}}) - \frac{n}{n + 1}c_{1}^{B}(M, \mathcal{F_{\xi}})^{2})\wedge\frac{\sigma(t)^{n - 2}}{(n - 2)!}\wedge\eta  \\& = \int_{M}\{\tr(R^{T}(t)\wedge R^{T}(t)) - \frac{1}{n + 1}\tr R^{T}(t)\wedge \tr R^{T}(t)\}\wedge \frac{\sigma(t)^{n - 2}}{(n - 2)!}\wedge\eta \\& = \int_{M} |R^{T}(t)|^{2} - (S^{T}(t))^{2} - \frac{n + 2}{n + 1}(|\rho^{T}(t)|^{2} - (S^{T}(t))^{2}) \frac{\sigma(t)^{n}}{n!}\wedge\eta 
 \end{split}
 \end{equation}
 where $S^{T} =  (h^{T})^{i\bar{j}}(h^{T})^{k\bar{l}}R_{i\bar{j}k\bar{l}}$  
 . Locally, set 
 \[ Q(t)_{i\bar{j}k\bar{l}}  =  R^{T}(t)_{i\bar{j}k\bar{l}} - \frac{S^{T}(t)}{n(n + 1)}(h^{T}(t)_{i\bar{j}}h^{T}(t)_{k\bar{l}} + h^{T}(t)_{i\bar{l}}h^{T}(t)_{k\bar{j}}).\ \]
 By direct calculation, we have
 \begin{equation}\label{6.3}
 |Q(t)|^{2} = |R^{T}(t)|^{2} - \frac{2(S^{T}(t))^{2}}{n(n+1)}.
 \end{equation} 
 Combining (\ref{6.2}) and (\ref{6.3}), we have 
  \begin{equation}\label{6.4}
 \begin{split}
 &(2\pi)^{2}\int_{M}(2c_{2}^{B}(M, \mathcal{F_{\xi}}) - \frac{n}{n + 1}c_{1}^{B}(M, \mathcal{F_{\xi}})^{2})\wedge\frac{\sigma(t)^{n - 2}}{(n - 2)!}\wedge\eta  \\& = \int_{M}(|Q(t)|^{2} + \frac{n + 2}{n(n + 1)}(S^{T}(t) + n)^{2} - \frac{n + 2}{n + 1}|\rho^{T}(t) + \sigma(t)|^{2})\frac{\sigma(t)^{n}}{n!}\wedge\eta 
 \\& \ge \int_{M} (|Q(t)|^{2} - \frac{n + 2}{n + 1}|\rho^{T}(t) + \sigma(t)|^{2})\frac{\sigma(t)^{n}}{n!}\wedge\eta 
 \end{split}
 \end{equation}
 By the equation (\ref{MA}), We know 
 \begin{equation}\label{6.5}
 \rho^{T}(t) = td\eta - \sigma(t).
 \end{equation}
 Combining (\ref{6.4}) and (\ref{6.5}), let $t \to 0$ we have 
 \begin{equation}
 \begin{split}
&(2\pi)^{2}\int_{M}(2c_{2}^{B}(M, \mathcal{F_{\xi}}) - \frac{n}{n + 1}c_{1}^{B}(M, \mathcal{F_{\xi}})^{2})\wedge\frac{(- c_{1}^{B}(M, \mathcal{F_{\xi}}))^{n - 2}}{(n - 2)!}\wedge\eta \\& = \lim \limits_{ t \to 0}(2\pi)^{2}\int_{M}(2c_{2}^{B}(M, \mathcal{F_{\xi}}) - \frac{n}{n + 1}c_{1}^{B}(M, \mathcal{F_{\xi}})^{2})\wedge\frac{\sigma(t)^{n - 2}}{(n - 2)!}\wedge\eta \\& \ge  \lim \limits_{ t \to 0}\int_{M} (|Q|^{2} - \frac{n + 2}{n + 1}|\rho^{T}(t) + \sigma(t)|^{2})\frac{\sigma(t)^{n}}{n!}\wedge\eta \ge 0.
\end{split}
\end{equation}
\end{proof}

\subsection*{Acknowledgement}
The author thank his supervisor, Professor Xi Zhang, for leading me to the research of this problem and offering many helpful discussions. The author also thanks Huang,liding for conversations about PDE's problems in this paper and Wu,di for conversations about Sasakian geometry. This article had been submitted to the international journal of math on August 13. The Manuscript Number is IJM-D-21-00136. The author supported by the National Key R and D Program of China 2020YFA0713100.

\end{document}